\documentclass{article}
\usepackage[utf8]{inputenc}
\usepackage{amsmath}
\usepackage{amssymb}
\usepackage{amsthm}
\usepackage{hyperref}
\usepackage{mathtools}
\usepackage{tikz-cd}

\usepackage{tikz}
\usetikzlibrary{matrix}
\hypersetup{
    colorlinks,
    citecolor=black,
    filecolor=black,
    linkcolor=black,
    urlcolor=black
}

\title{The K-theory of boundary C*-algebras of symmetric spaces}
\date{\vspace{-5ex}}

\newcommand{\Q}{\mathbb{Q}}
\newcommand{\R}{\mathbb{R}}
\newcommand{\C}{\mathbb{C}}
\newcommand{\N}{\mathbb{N}}
\newcommand{\Z}{\mathbb{Z}}
\renewcommand{\H}{\mathbb{H}}

\newtheorem{theorem}{Theorem}[section]
\newtheorem{proposition}[theorem]{Proposition}
\newtheorem{lemma}[theorem]{Lemma}
\newtheorem{corollary}[theorem]{Corollary}
\newtheorem{definition}[theorem]{Definition}
\newtheorem{example}[theorem]{Example}

\begin{document}
\author{Torstein Ulsnaes}

\maketitle

We compute the K-theory of a collection of C*-algebras, which we refer to as boundary C*-algebras, arising as the crossed product C*-algebras of lattice actions on the Furstenberg boundaries of symmetric spaces of noncompact type. Along the way we show that some of these boundary C*-algebras are isomorphic but not spatially isomorphic. Several other examples of isomorphic crossed product C*-algebras constructed from very different dynamical systems can be found in \cite{phillips_examples_2007}, \cite{geffen_note_2024}, \cite{robertson_steger}.

I'm grateful to Bram Mesland and Yuezhao Li for pointing me to the work of Geffen and Kranz \cite{geffen_note_2024} which also computes the K-theory of boundary C*-algebras for lattices in $\H^n_\R$ where $n =  3$. The methods in the present paper were developed independently of theirs as part of my Ph.D. project and work in any dimension.

\section{Introduction}
\label{sec-introduction}

The boundary C*-algebras, which we will define shortly, have been studied by several authors (\cite{geffen_note_2024}, \cite{fuchs_cyclic_2007}, \cite{emerson_euler_2006}, \cite{mesland_hecke_2018}, \cite{robertson_steger}) and arise naturally when dealing with boundary extensions of symmetric spaces. Explicitly, if $X$ is a symmetric space of noncompact type and $\Gamma \subset \mathrm{Iso}(X)^0$ is a lattice, the boundary extension
\begin{equation}
    \label{eq-boundary-extension}
    0 \to C_0(X) \to C(\overline{X}) \to C(\partial X) \to 0
\end{equation}
given by the geodesic compactification $\overline{X}$ of $X$, yields an extension of crossed product C*-algebras

    $$0 \to C_0(X) \rtimes_r\Gamma \to C(\overline{X}) \rtimes_r\Gamma \to C(\partial X) \rtimes_r \Gamma \to 0.$$
In the rank $1$ case, the boundary $\partial X$ is a homogeneous space $$\partial X \simeq G/P$$
where $P \subset G = \mathrm{Iso}(X)_0$ is a minimal parabolic subgroup (Definition \ref{def-parabolic-subgroup}). For higher rank symmetric spaces the geodesic boundary is no longer a homogeneous space. With $P_i \subset G$ a parabolic subgroup, the spaces $G/P_i$, which are referred to as generalized flag manifolds, can still be thought of as "boundaries" in the abstract sense of Furstenberg \cite{furstenberg_poisson_1963} (see also \cite[Deﬁnition I.6.6]{a_borel_compactifications_2006}). In a more concrete sense, the maximal Furstenberg boundary appears as the quotient space in an equivariantly semi-split extension
$$0\to C_0(Y) \to C(\overline{X}^F) \to C(G/P_0)\to 0$$
mimicking the boundary extension \eqref{eq-boundary-extension} (see \cite{ulsnaes_boundary_2024}). The space $G/P_0$ is closely linked to the space of harmonic functions on $X$, as was shown in \cite{furstenberg_poisson_1963}. Here $\overline{X}^F$ is the so called maximal Furstenberg compactification of the symmetric space $X$ and  $Y$ is $\overline{X}^F$  removed a (unique) compact $G$-orbit $G/P_0 \subset \partial \overline{X}^F$ isomorphic to $G/P_0$  (see \cite[I.6]{a_borel_compactifications_2006}).

We will in the present work give a method to compute the K-theory of $C(G/P) \rtimes_r\Gamma$ from the K-theory of a certain vector bundle. In the case where the group $G$ has real rank 1 the base of the bundle is identified with the unit sphere bundle of the associated locally symmetric space $\Gamma \backslash X$. Assuming $\Gamma$ is cocompact and the cohomology groups $H^*(\Gamma \backslash X)$ are torsion free, the Chern map lets us determine the K-theory of these boundary C*-algebra from the cohomology of $\Gamma \backslash X$.

This paper is organized as follows. In Section \ref{sec-mostow-and-boundary-algebras} we will motivate the study of boundary C*-algebras from a geometric point of view, using the classical Mostow rigidity theorem as our guide. Section \ref{sec-boundary-algebras-and-tangent-bundles} contains the main results of this work, where the boundary C*-algebras are related to sphere bundles on the locally symmetric space $\Gamma \backslash X$. 

More precisely, we show (Proposition \ref{prop-boundary-kk-equivalence}) that under mild conditions the C*-algebras $C(G/P)\rtimes_r\Gamma$ share the K-theory of the total space of a certain vector bundle $\Gamma \backslash G\times_{K_i} V \to \Gamma \backslash G / K_i$. In the rank 1 case this can be strengthened to a KK-equivalence, and we can show that $\Gamma \backslash G / K_i$ can be identified with the unit tangent bundle of $\Gamma \backslash X$ (Lemma \ref{lem-sphere-bundle-iso}). 

In Section \ref{sec-torsion-free-cohomology} we put our theory to use and compute the K-theory of the boundary C*-algebras for cocompact lattices $\Gamma \subset G$ of symmetric spaces for which $\Gamma \backslash X$ is oriented and has no torsion in its cohomology. We do this by first computing the cohomology of the unit sphere bundle of $\Gamma \backslash X$ (Lemma \ref{lem-cohomology-unit-tangent-bundle}) and then following the torsion part of the cohomology through the Atiyah-Hirzebruch spectral sequence.

\section{Notation}

For a comprehensive introduction to the theory of symmetric spaces we refer the reader to \cite{helgason_differential_1979}. Compactifications of symmetric spaces are covered in \cite{a_borel_compactifications_2006} where the relevant sections are I.2 dealing with the geodesic compactification and I.1 which covers parabolic subgroups of Lie groups and their Langlands decomposition.

Throughout the present paper we will use the following notation, possibly adding more conditions when needed. We denote by $X$ a symmetric space of noncompact type, that is, a Riemannian symmetric space with non-positive sectional curvature and no euclidean factors in its de Rham decomposition. The dimension $n = dim(X)$ of $X$ will be denoted by $n$.  \textit{All symmetric spaces will be assumed to be of noncompact type}, which are precisely those symmetric spaces whose sectional curvature is non-positive and have no euclidean factors in their de Rham decomposition. We denote by $G = \mathrm{Iso}(X)^0$ the connected component of the isometry group of $X$ which is known to be a semisimple Lie group with trivial center. By $K$ we denote a maximal compact subgroup of $G$ and $\Gamma \subset G$  denotes a torsion-free lattice in $G$, which is  a torsion-free discrete subgroup such that $\Gamma \backslash G$ has finite volume with respect to the restricted Haar measure of $G$ to any fundamental domain of $\Gamma$ in $G$.

A minimal parabolic subgroup of $G$ will be denoted by $P_0$ or simply by $P$ if the rank of $G$ is $1$, in which case all parabolic subgroups are minimal. 

We recall that all maximal compact subgroups of $G$ arise as the stabilizer subgroups of some point $x_0\in X$ and, given a suitable $G$-invariant metric on the homogeneous space $G/K$, where $K = \mathrm{Stab}_G x_0$, we get an isometry of Riemannian spaces \cite[Theorem IV.3.3 (i)]{helgason_differential_1979}

    $$X \simeq G/K \qquad x = gx_0 \mapsto gK.$$
We denote by 
$$G = KAN$$
the Iwasawa decomposition of $G$, where $A$ is a connected abelian and $N$ is a nilpotent group. For a minimal parabolic subgroup $P_0 \subset G$ we write
$$P_0 = MAN = N\rtimes MA$$
for its Langlands decomposition where $M = P_0\cap K$ is a maximal compact subgroup of $P_0$. Recall that the dimension $\mathrm{dim}(A)$ is called the (real) \textit{rank} of $X$, which is also the largest dimension of any totally geodesic flat submanifold of $X$.

\section{Mostow rigidity and boundary C*-algebras}
\label{sec-mostow-and-boundary-algebras}

Let $x\in X$ and $v\in T_xX$ be any vector in the tangent space at $x$. We denote by  $\gamma_{v}^x: \R \to X$ be the unique geodesic satisfying
\begin{align*}
    \gamma_v^x(0) &= x \\
    (\gamma_v^x)'(0) &= v.
\end{align*}
Let 

    $$S  = \{\gamma_v^{x} ~|~ v\in S^{n-1} \subset T_{x}X \}.$$
We define an equivalence relation on $S$ by saying
    
    $$\gamma_v^x \sim \gamma_w^y \quad \Leftrightarrow \quad  \sup_{t\in R^+} d(\gamma_v^x(t), \gamma_w^y(t)) < \infty.$$
The space $S$ carries a $G$-action given by 
$$g\cdot \gamma_v^x := \gamma_{g_*v}^{gx}.$$
which is easily seen to preserve the above equivalence relation. We then define the geodesic boundary of $X$ to be
$$\partial X = S/\sim.$$

For a fixed $x_0\in X$ and an arbitrary $\gamma_v^x$ with $v\in S^{n-1} \subset T_xX$ there is a unique geodesic $\gamma_{x_0}^{v_0}$ such that $\gamma_x^v \sim \gamma_{x_0}^{v_0}$. We thus get an identification of $\partial X$ with the subset $S^{n-1} \subset T_{x_0}X$ and endow $\partial X$ with the topology inherited from identification. With respect to this topology the action of $G$ on $\partial X$ is continuous. For the proof of these claims we refer the reader \cite[Section I.2]{a_borel_compactifications_2006}.

In the context of Lie groups, parabolic subgroups are defined as follows:
\begin{definition}[Parabolic subgroup]
\label{def-parabolic-subgroup}
A parabolic subgroup of $G$ is a stabilizer in $G$ of a point $v\in \partial X$. A parabolic subgroup $P_0 \subset G$ is called minimal if it contains no (non-trivial) parabolic subgroups.
\end{definition}

We mention without proof that just as for Borel subgroups, any parabolic subgroup of $G$ is conjugate to a unique parabolic subgroup containing a fixed minimal parabolic subgroup (see \cite[p.32]{a_borel_compactifications_2006}).

\begin{definition}[Geodesic and Furstenberg boundary]
    The $G$-space $\partial X$ is called the geodesic boundary of $X$. The spaces $G/P$ for $P$ a minimal parabolic subgroup of $G$ are called the Furstenberg boundaries of $X$ (or $G$). If $P = P_0$ is minimal, we call $G/P_0$ a maximal Furstenberg boundary.
\end{definition}

Since all minimal parabolic subgroups of $G$ are conjugate to one another, all maximal Furstenberg boundaries of $G$ are ($G$-equivariantly) isomorphic. Similarly all parabolic subgroups  $P \subset G$ are conjugate to a unique parabolic subgroup containing a given minimal parabolic subgroup $P_0$. 

For these reasons we will refer to $G/P_0$  as \textit{the} maximal Furstenberg boundary. Let us quickly state the main property of the geodesic and Furstenberg boundary which we will need:

\begin{proposition}[{\cite[Prop. I.2.25]{a_borel_compactifications_2006}}]
    \label{prop-properties-geodesic-comp}
    Let $v\in \partial X$ and let $P_i = \mathrm{Stab}_G v$ be the associated parabolic subgroup. Then the natural $G$-equivariant map
        $$G/P_i \to \partial X\qquad gP_i \mapsto g v$$
    determines an imbedding of the  Furstenberg boundary $G/P_i$ into the geodesic boundary $\partial X$ which is surjective if and only if $X$ has rank 1.
\end{proposition}

In summary, any Furstenberg boundary $G/P_i$ can always be identified with a closed $G$-orbit in the geodesic boundary $\partial X$ which agrees with $\partial X$ if and only if $\mathrm{rk}(X) = 1$. We mention that even though the imbedding $G/P_i \to \partial X$ is not unique, there is a canonical choice given by the $G$-orbit of the "barycenter" of a positive Weyl chamber.

Let us now introduce our main objects of study:

\begin{definition}[Boundary C*-algebras]
    Let $\Gamma \subset G$ be a lattice and $P_i \subset G$ a parabolic subgroup. We call the crossed products C*-algebras
    $$C(G/P)\rtimes_r\Gamma$$
    boundary C*-algebras associated with the lattice $\Gamma$. We refer to $C(G/P_0)\rtimes_r\Gamma$ as the maximal boundary C*-algebra.
\end{definition}

To get an idea why these boundary C*-algebras are of interest from a geometric point of view, we now make a short digression and recall the classical Mostow rigidity theorem (\cite[Theorem 24.1]{mostow_strong_1973}). The theorem states that a large family of locally symmetric spaces are determined (essentially) up to isometry by their fundamental groups, giving an affirmative answer to the Borel conjecture for these spaces. We will use the following formulation of the theorem, which can be deduced from \cite[Theorem 15.1.1]{morris_introduction_2015} (see also exercise 8 on page 308 of the same reference).

\begin{theorem}[Mostow]
    \label{th-mostow}
    Let $X$ be a symmetric space of noncompact type of rank 1 with $\mathrm{dim}(X) \geq 3$, and let $\Gamma, \Gamma' \subset \mathrm{Iso}(X)^0$ be two (torsion-free) lattices. Then $\Gamma \backslash X$ and $\Gamma' \backslash X$ are isometric up to scaling of the metric on each de Rham factor if and only $\Gamma = \Gamma'$ as abstract groups.
\end{theorem}

The reader may be more familiar with the case where $X = \H^n_\R$ with $n\geq 3$ which has only a single de Rham factors, and so the theorem simplifies accordingly. 

An elegant proof Theorem \ref{th-mostow}, very much in the spirit of the original proof of Mostow, uses the construction of an isometry of symmetric spaces from a homeomorphism of their geodesic boundaries by means of the so called \textit{barycenter extension method}. We refer to interested reader to \cite[Section 2]{matic_recent_2003} for an excellent overview of this method together with several applications to ergodic theory. For the purpose of motivating the study of boundary C*-algebras, we only need to know the following:

If we are given an isomorphism $\phi: \Gamma \to \Gamma'$ of irreducible lattices in a rank 1 symmetric space $X$ of noncompact type, and a homeomorphism $f: \partial X \to \partial X$ which is $\Gamma-\Gamma'$-equivariant in the sense that

$$f(\gamma x) = \phi(\gamma) f(x) \qquad \text{for all } x\in X \text{ and all } \gamma \in \Gamma,$$
then the barycenter extension method produces an isometry $\hat{f}: X\to X$ which will also be $\Gamma-\Gamma'$-equivariant and so descends to an isometry of the quotients 

    $$\hat{f}: \Gamma \backslash X \to \Gamma' \backslash X.$$
One could thus think of the boundary dynamical system $(\partial X, \Gamma)$ as the classifying datum for the locally symmetric spaces $X/\Gamma$, rather than the fundamental groups $\Gamma$ alone. 

Though this seems to only add more data to our classification, it does allow us to pose the following question: Do the corresponding reduced crossed product C*-algebras
\begin{equation}
    \label{eq-boundary-dyn-sys}
    C(\partial X)\rtimes_r\Gamma = C^*_r(\partial X \rtimes \Gamma)
\end{equation}
entirely classify the locally symmetric spaces $\Gamma \backslash X$?

This seems quite optimistic (and answer is indeed no), since in general much information about the underlying dynamical system is lost when passing to its associated crossed product C*-algebra  \cite{phillips_examples_2007}. However one would still expect that the boundary C*-algebras retain some geometric information of the underlying locally symmetric space $\Gamma \backslash X$, at least under sufficiently stringent conditions on $X$ or $\Gamma$. 

In this paper we will give a geometric interpretation of the boundary C*-algebras and produce a method to compute their K-theory. As a corollary we show that in many cases two boundary C*-algebras can be isomorphic though their underlying locally symmetric spaces are non-isometric, adding to the family of examples computed in \cite{geffen_note_2024}.

\section{Boundary C*-algebras and the unit tangent bundle of $\Gamma \backslash X$}
    \label{sec-boundary-algebras-and-tangent-bundles}
    The boundary C*-algebras of equation \eqref{eq-boundary-dyn-sys} are closely related to the unit tangent bundles of the locally symmetric spaces $\Gamma \backslash X$ (see Proposition \ref{prop-main-prop}). The analogy has been noted by several authors (for instance \cite{fuchs_cyclic_2007} and  \cite{emerson_euler_2006}), but when this correspondence is made explicit, an assumption of Spin${}^c$-structure is added on the locally symmetric space $\Gamma \backslash X$. We remove this assumption at the expense of having to work with twisted K-theory.
    
    Let us start with a lemma:

    \begin{lemma}
        \label{lem-sphere-bundle-iso}
        Assume $\mathrm{rk}(X) = 1$ and let $\Gamma \subset G$ be a torsion free lattice. Let $T^1\Gamma \backslash X$ be the unit tangent bundle of $\Gamma \backslash X$. Then there is an isomorphism of sphere bundles:
        \begin{center}
            \begin{tikzcd}
                K/M \arrow[r] \arrow[d] & \Gamma \backslash G /M   \arrow[r] \arrow[d]& \Gamma \backslash G /K \arrow[d] \\
                S^{n-1} \arrow[r]  & T^1 \Gamma \backslash  X   \arrow[r] &  \Gamma \backslash  X.
            \end{tikzcd}
        \end{center}
    \end{lemma}
    \begin{proof}
        By Proposition \ref{prop-properties-geodesic-comp} we have a $G$-equivariant homeomorphism
            $$\partial X = G/P$$
        where $P\subset G$ is the minimal parabolic subgroup of elements in $G$ fixing a point $v_0 \in \partial X$. Let $x_0\in X$ be any point, $K = \mathrm{Stab}_Gx_0$ and identify $\partial X$ with $S^{n-1} \subset T_{x_0}X$. We denote by $v_0$ also the image of $v_0$ in $T_{x_0}X$ under this identification.

        When $X$ has rank 1, the continuous action of $G$ on $T^1X$ is transitive (since in this case $K = \mathrm{Stab}_Gx_0$ acts transitively $S^{n-1} \subset T_{x_0}X$). Thus there is  a $G$-equivariant homeomorphism 
        $$ T^1X \simeq G/\mathrm{Stab}_G(x_0, v_0)$$
        where
            $$\mathrm{Stab}_G(x_0, v_0) = \mathrm{Stab}_G(x_0) \cap \mathrm{Stab}_G(v_0) = K\cap P = M.$$
        We readily get
        $$G/M \simeq T^1X \qquad$$
        as (left, non-principal) $G$-bundles over $X = G/K$. It follows that for a closed subgroup $H\subset G$ we have
        $$T^1 (H\backslash X) = H\backslash T^1X = H\backslash G/M.$$
        Setting $H = \Gamma$ completes the proof.
    \end{proof}

    \textbf{Remark:} In case the space $X$ has higher rank, the action of $G$ on $T^1X$ is no longer transitive, thus Lemma \ref{lem-sphere-bundle-iso} does not hold. However if $M = P_0\cap K$ for a minimal parabolic subgroup $P_0 \subset G$, the $G$-bundle
    $$K/M \to G/M \to G/K = X$$
    is a subbundle of $T^1 X$ consisting of unit vectors pointing in the direction of maximal divergence of geodesic rays. Equivalently the subbundle which at each $x_0 \in X$  is spanned by the vectors  $v \in G/P_0\subset \partial X \simeq S^{n-1} \subset T_{x_0}X$ pointing towards a the canonical imbedded image of $G/P_0$ in $\partial X$.
    
    Recall that a Kirchberg algebra is a separable, nuclear, simple, purely infinite C*-algebra. The following proposition tells us when the boundary C*-algebras are Kirchberg algebras:
    
    \begin{proposition}[{\cite[Prop. 3.4]{delaroche_purely_1997}}]
        \label{prop-kirchberg}
        Let $P_0\subset G$ be a minimal parabolic subgroup. Then the C*-algebras $$C(G/P_0)\rtimes_r\Gamma$$
        are Kirchberg algebras in the UCT class.
    \end{proposition}

    \textbf{Remark} Minimality of $P_0$ is essential in Proposition \ref{prop-kirchberg}. 
    If $P$ is not minimal, then the action of $\Gamma$ on $G/P$ is not amenable (see \cite[Cor. 4.3.7]{zimmer}) hence by \cite[Prop. 1.4.4]{delaroche_purely_1997} the crossed product
    $$C(G/P)\rtimes_r\Gamma$$
    is not nuclear (it is only exact\footnote{Since the quotient map $G/P_0 \to G/P$ gives an inclusion $C(G/P)\rtimes_r\Gamma \subset C(G/P_0)\rtimes_r\Gamma$.})

    As noted in Proposition \ref{prop-properties-geodesic-comp} when $X$ has rank 1, there is a $G$-equivariant isomorphism $\partial X\simeq G/P$. Thus in this case (and only this) the crossed product 
        $$C(\partial X)\rtimes_r\Gamma = C(G/P) \rtimes_r\Gamma$$
    are all Kirchberg algebras in the UCT class. In higher ranks $C(\partial X) \rtimes_r\Gamma$ is not simple: The action of $G$ on $C(\partial X)$ has fundamental domain equal to any choice of closed "positive Weyl chamber at infinity" $\overline{A^+}(\infty) \subset \partial X$ (this could be deduced from the proof of Prop. I.2.16 in \cite{a_borel_compactifications_2006} where the author shows $\partial X = K \overline{A^+}(\infty )$ and the fact that $K$-orbits and $G$-orbits in $\partial X$ agree). This in turn implies that $C(\partial X)$ has non-trivial $\Gamma$-invariant ideals.

    We are now ready to state the following proposition:

    \begin{proposition}
        \label{prop-boundary-kk-equivalence}
        Let $P_i \subset G$ be a parabolic subgroup with Langlands decomposition $P_i = M_i A_i N_i$. Denote by $K_i =  P_i\cap K = M_i\cap K$ its maximal compact subgroup. Assume the Baum-Connes conjecture with coefficients in $C(\Gamma \backslash G)$ holds for $P_i$. Then  

        $$K_*(C(G/P_i)\rtimes_r\Gamma) = K_*( C_0(\Gamma \backslash G \times_{K_i} V))$$
        where $V = T_{e}X$ with the action of $K_i$ induced by conjugation and $\Gamma \backslash G \times_{K_i} V$ is the orbit space of $K_i$ under the diagonal action on $\Gamma \backslash G \times V$.
    \end{proposition}
    
    \begin{proof} 
        Using Green's symmetric imprimitivity theorem we get a Morita equivalence 

        \begin{align*}
            C(G/P_i) \rtimes_r\Gamma & \sim_{Mor} C_0(\Gamma \backslash G) \rtimes_rP_i 
        \end{align*}
        Now we follow the argument in Section 7 of \cite{chabert_echterhoff_nest} where the authors show that 

        \begin{align*}
            K_*(C(\Gamma \backslash G)\rtimes_r P_i) &= K_*(C(\Gamma \backslash G) \otimes C_0(V) \rtimes K_i)
        \end{align*}
      
        Finally, since $\Gamma$ is torsion free, the action of $K_i$ is free, so by Green's Moria equivalence we get
        $$C(\Gamma \backslash G ) \rtimes_r C_0(V) \rtimes_r K_i \sim_{Mor}C_0(\Gamma \backslash G \times_{K_i} V).$$
    \end{proof}

    As the UCT class is closed under crossed products by amenable actions, Proposition \ref{prop-boundary-kk-equivalence} gives a KK-equivalence 

    $$C(G/P_0) \rtimes_r \Gamma \sim_{KK} C(\Gamma \backslash G \times_{M} V).$$

    However, for a non-minimal parabolic subgroups such a KK-equivalence may not hold,  but it does hold up to a direct summand by a C*-algebra with trivial K-theory:

    \begin{lemma}
        With the notation of Proposition \ref{prop-boundary-kk-equivalence}, we have a KK-equivalence
        $$C(G/P_i) \rtimes_r \Gamma \simeq_{KK} C(\Gamma \backslash G \times_{K_i} V) \oplus \ker(\hat{\gamma})$$
        where $\hat{\gamma} \in KK(C(G/P_i) \rtimes\Gamma, C(G/P_i)\rtimes_r\Gamma)$ is Kasparovs $\hat{\gamma}$ element defined in \cite[Theorem 2, Sec. 6]{ranicki_k-theory_1995}.
    \end{lemma}

    \textbf{Remark: } Kasparov showed that the Baum-Connes map factors through the image of the "multiplication by $-\otimes \hat{\gamma}$" map on  $K_*(C(G/P)\rtimes_rP_i)$ hence our assumptions ensure this multiplication map must be an isomorphism. However, this is weaker than requiring that $\hat{\gamma} = 1 \in KK(C(\Gamma \backslash G) \rtimes_r P_i, C(\Gamma \backslash G) \rtimes_r P_i)$, but it does imply that $\ker(\hat{\gamma})$ have trivial K-theory. If we knew that $\ker(\hat{\gamma})$ (or  $C(\Gamma \backslash G)\rtimes_r P_i$) was in the UCT class, this would force $\ker(\hat{\gamma})$ to be KK-equivalent to $0$, but it is unclear to the author if this is the case. 
    
    \begin{proof}
        We will need the following facts: The category $KK$ is idempotent complete, meaning every idempotent $$\alpha \in KK(A, A)$$ admits a (categorical) kernel.\footnote{This can be shown explicitly realizing the map $a$ as a *-homomorphism on a C*-algebra KK-equivalent to $A$, but it also follows from the fact that KK is a triangulated category with countable direct sums, hence \cite[Prop. 1.6.8]{neeman} guarantees it is idempotent complete.}

        It is shown in \cite[Theorem 2(1), Sec. 6]{ranicki_k-theory_1995} that there are elements 

        $$\alpha \in KK( C(\Gamma \backslash G) \rtimes_r P_i, C(\Gamma \backslash G \times_{K_i} V))$$
        $$\beta \in KK(C(\Gamma \backslash G \times_{K_i} V), C(\Gamma \backslash G) \rtimes_r P_i)$$
        satisfying\footnote{we are tacitly identifying $\C_V$ with $C_0(V)$.}

            $$\alpha \otimes \beta = id \qquad \beta \otimes \alpha = \hat{\gamma}.$$
        Thus $\alpha$ is $KK$-equivalent to a retract of $ C(\Gamma \backslash G) \rtimes_r P_i$ to $C(\Gamma \backslash G \times_{K_i} V)$.
        It follows that
        
        $$C(\Gamma \backslash G) \rtimes_r P_i \simeq_{KK} C(\Gamma \backslash G \times_{K_i} V) \oplus \ker(\hat{\gamma}).$$
    \end{proof}

    By Proposition \ref{prop-boundary-kk-equivalence} we have reduced the problem of computing the K-theory of the boundary C*-algebras to that of computing the topological K-theory of the total spaces of a the vector bundles 
    $$V\to \Gamma \backslash G\times_{K_i} V\to \Gamma \backslash G / K_i.$$

    The next corollary covers the "nice" case where the action of $K_i$ is spinnor, meaning it factors through an action of $Spin_n^\C$ on $V$.

    \begin{corollary}
        \label{cor-boundary-kk-equivalence}
        Keeping the notation of Proposition \ref{prop-boundary-kk-equivalence}, assume the action of $K_i$ on $V$ factors through an action of $Spin^c_n$ on $V$. Then
        $$K_*(C(G/P_i) \rtimes_r \Gamma) = K_{*+n}(C_0(\Gamma \backslash G / K_i))$$
    \end{corollary}
    \begin{proof}
        By \cite[Lemma 1, p. 546]{kasparov_operator_1981} we know $$C_0(V) \sim_{KK^{Spin^c_n}} C_0( \R^n)$$ 
        where $\R^n$ has the trivial $Spin^c_n$-action. By assumption the action of $K_i$ factors through a map
            $$K_i\to Spin^c_n \to SO(V).$$
        Since the induced restriction map
        $$KK^{Spin^c_n} \to KK^{K_i}$$
        is functorial and preserves the class of the unit in $KK$-theory we readily get

        $$C_0(V) \sim_{KK^{K_i}} C_0(\R^n).$$
        Thus by the proof of Proposition \ref{prop-boundary-kk-equivalence}
        \begin{align*}
             K_i(C(G/P_i) \rtimes_r\Gamma)  & =  K_i(( C_0(\Gamma \backslash G) \otimes C_0(V)) \rtimes_r{K_i} )  \\ & = K_i(  C_0(\Gamma \backslash G) \rtimes_r{K_i} \otimes C_0(\R^n)) = K_{i+n}( C_0(\Gamma \backslash G / {K_i})).
        \end{align*}
        Where the last equality follows by Morita equivalence (the action of $K_i$ on $\Gamma \backslash G$ is free and proper as $\Gamma$ is torsion free) and Kasparov's Bott periodicity.
    \end{proof}

    \textbf{Remark:} It is well known that any representation which preserves a complex structure on $V$ will factor through $Spin_r^\C$. Thus Corollary \ref{cor-boundary-kk-equivalence} applies to all hermitian locally symmetric spaces like those covered by $\H^n_\C$. Similarly one can show that all other rank 1 symmetric  spaces except for the real hyperbolic spaces, the action of $M$ on $V$ factors through $Spin^c_n$. A good reference for rank 1 symmetric spaces is section 2.2 of \cite{quint_overview_2006}.
    
    In the case of $\H_\R^n$, we have $P_0\cap K = M = SO(n-1)$ and the action of $M$ on $V$ is equivalent to the action of $M$ on $\R^n$ given by
    $$m\cdot (x_1, x_2, \dots, x_n) := (x_1, m(x_2, \dots, x_m))$$
    where $m$ acts by the usual matrix action of $SO(n-1)$ on $\R^{n-1}$ on the right hand side (see \cite{quint_overview_2006}). This action does not factor through $\mathrm{Spin^c_n}$ when $n \geq 4$ so is not covered by Corollary \ref{cor-boundary-kk-equivalence}\footnote{To see this, note that $\pi_1(Spin^c(n))= \Z$ for $n\geq 3$. The inclusion $SO(n) \to SO(n+1)$ induces an isomorphism of fundamental groups for $n\geq 3$ with $\pi_1(SO(n)) = \Z_2$ it is clear that the identity $\pi_1(SO(n)) = \Z_2 \xrightarrow[]{id} \Z_2 = \pi_1(SO(n+1))$ cannot factor through a map $\pi_1(Spin^c(n)) = \Z \to \Z_2 = \pi_1(SO(n+1))$.}. 
   
    To tackle the non-Spin${}^c$-case, we first recall that an oriented bundle over a smooth manifold $Y$ admits a Spin${}^c$-structure if and only if $$W_3 = 0 \in H^3(Y, \Z)$$
    where $W_3$ is the third integral Stiefel-Whitney class of the bundle. In case $W_3 \neq 0$ the following twisted version of Thom-isomorphism theorem (see for instance \cite[Theorem 3.5]{carey_bai_ling_2008}) gives the K-theory for orientable boundary C*-algebras in terms of the K-theory of the associated locally symmetric space. Recall that the twisted K-group $K_*(X, w) = K_*(A_w)$ where $w\in H^3(X, \Z)$ and $A_w$ is the stable continuous trace C*-algebra with Dixmier-Douady invariant $w$ and spectrum $X$.

    \begin{proposition}
        \label{prop-main-prop}
        Let $P_i \subset G$ be a parabolic subgroup satisfying the conditions in Proposition \ref{prop-boundary-kk-equivalence}. Then 
        $$K_i(C(G/P_i)\rtimes_r\Gamma) =  K_{i+n}(C_0(\Gamma \backslash G /K_i), W_3)$$
        where the right hand side is the twisted K-theory, twisted by the Stiefel-Whitney class of the vector bundle 
        $$V\to \Gamma \backslash G \times_{K_i} V \to \Gamma \backslash G/ K_i.$$
    \end{proposition}
    \noindent 
    We remark that

        $$\Gamma \backslash G \times_{K_i} V= \pi^*(\Gamma\backslash G\times_K V)$$
    where $\pi$ is the fibration 
    \begin{equation}
        \label{eq-q-map}
        \pi: \Gamma \backslash G /K_i \to \Gamma \backslash G /K.
    \end{equation}
    Thus by naturality of the Stiefel-Whitney class, if $\Gamma \backslash G/K$ admits a $Spin^c$-structure, so do the bundles
    $$V \to \Gamma \backslash G\times_{K_i} V \to \Gamma \backslash G / K_i$$
    for all $i$. In particular 

    \begin{corollary}
        \label{cor-spinc-kk-equiv}
         Keeping the notation of Proposition \ref{prop-main-prop}, assume further that $\Gamma \backslash X$ admits a  $Spin^c$-structure, then 
         $$K_i(C(G/P_i)\rtimes_r\Gamma) =  K_{i+n}(C_0(\Gamma \backslash G /K_i)).$$
    \end{corollary}

    In general a $Spin^c$-structure on $\Gamma\backslash G /K$ may be a stronger  requirement than we need for the equivalence of Corollary \ref{cor-spinc-kk-equiv} to hold, as we only need $\pi^*(W_3) = 0 \in H^3(\Gamma \backslash G / K_i, \Z)$ where $\pi$ the quotient map of equation \eqref{eq-q-map}.
    Let us look a little closer at the case where $G$ as rank 1. In this case $\Gamma \backslash G/M$ is the sphere bundle of $\Gamma \backslash G / K$ and (assuming it is orientable) we get a long exact Gysin sequence

    $$\cdots \to H^{r-n -2}(\Gamma \backslash X) \to H^r(\Gamma \backslash X)\xrightarrow{\pi^*} H^r(\Gamma \backslash G/ M) \to H^{r-n-1}(\Gamma \backslash X) \to \cdots$$
    we can deduce that $\pi^*: H^3(\Gamma \backslash X) \to H^3(\Gamma \backslash G/M)$ is injective if $n\geq 2$.

    Noting that $T\Gamma \backslash X = \Gamma \backslash G \times_K V$, we see that with these assumptions in the rank 1 case, $\Gamma \backslash G \times_M V$ is $Spin^c$ if and only if $\Gamma \backslash G / K$ is $Spin^c$. Hence we have no loss of generality by assuming $\Gamma \backslash X$ is $Spin^c$.
 
\section{The case of torsion free cohomology}
\label{sec-torsion-free-cohomology}
With Proposition \ref{prop-main-prop} and Corollary \ref{cor-spinc-kk-equiv} at our disposal we are ready to compute the K-theory of some boundary C*-algebras in terms of the cohomology groups of the associated locally symmetric space under the assumption that the locally symmetric space $\Gamma \backslash X$ has no torsion in its integral cohomology and admits a $Spin^c$-structure. It is well known that if $dim(\Gamma \backslash X) \leq 4$ then $\Gamma \backslash X$ admits a $Spin^c$-structure if and only if it is orientable, but in higher dimensions this is no longer true. Throughout this section we will assume $\Gamma$ is a cocompact torsion free lattice in $G$.

We start by computing the cohomology of the unit tangent bundle of $\Gamma \backslash X$:

\begin{lemma}
    \label{lem-cohomology-unit-tangent-bundle}
    Assume $X$ has rank 1 and denote by  $\chi: = \chi(\Gamma\backslash X)$ its Euler characteristic. Then

    $$\begin{matrix*}[l]
        H^i(\Gamma \backslash G/ M) = H^i(\Gamma \backslash X)  & 0\leq i < n-1 \\
        H^i(\Gamma \backslash G/ M) = H^{i-n + 1}(\Gamma \backslash X) &  n+1\leq i \leq 2n - 1.
    \end{matrix*}$$
    while if $\chi = 0$ we have
    \begin{align*}
        H^{n-1}(\Gamma \backslash G/M) & = H^{n-1}(\Gamma\backslash X)\oplus \Z \\
        H^{n}(\Gamma \backslash G / M) & = H^1(\Gamma\backslash X)\oplus \Z
    \end{align*}
    If $\chi = 1$ we have \footnote{It is known that the case $|\chi| = 1$ does not occur, but we add it here for completeness. See the discussion after Proposition \ref{prop-k-theory-boundary-algebras}}
    \begin{align*}
        H^{n-1}(\Gamma \backslash G/M) & = H^{n-1}(\Gamma\backslash X) \\
        H^{n}(\Gamma \backslash G / M) & = H^1(\Gamma\backslash X).
    \end{align*}
    If $\chi \neq 0$ and $\chi \neq 1$ we have
    
    \begin{align*}
        H^{n-1}(\Gamma \backslash G/M) & = H^{n-1}(\Gamma\backslash X)\oplus \Z \\
        H^{n}(\Gamma \backslash G / M) & = H^1(\Gamma\backslash X)\oplus \Z_{|\chi|}.
    \end{align*}
\end{lemma}

\begin{proof}
    Since $\Gamma \backslash X$ is assumed to be compact and orientable (as $G$ is connected), so is its unit sphere bundle $\Gamma \backslash G/ M$ (see Lemma \ref{lem-sphere-bundle-iso}) in the sense of \cite[p. 437-438]{hatcher_algebraic_2005}. We thus have a Gysin sequence 
    $$\cdots \to H^{i-n}(\Gamma \backslash X) \to H^i(\Gamma \backslash X)\to H^i(\Gamma \backslash G/ M) \to H^{i-n+1}(\Gamma \backslash X) \to \cdots$$
    Since $H^{i}(\Gamma \backslash X) = 0$ if $i< 0$ we get that  
    $$H^i(\Gamma \backslash G/ M) = H^i(\Gamma \backslash X) \qquad (0\leq i < n-1)$$
    and in a similarly vein
    $$H^i(\Gamma \backslash G/ M) = H^{i-n + 1}(\Gamma \backslash X) \qquad (n+1\leq i \leq 2n - 1).$$
    For the values $i = n, n-1$ we have an exact sequence
    \begin{align}
        \label{eq-h-n-and-h-n-1}
        0 & \to H^{n-1}(\Gamma \backslash X) \to H^{n-1}(\Gamma \backslash G / M) \to \Z \xrightarrow[\chi]{} \Z  \\
        & \to H^n(\Gamma \backslash G / M)\to H^1(\Gamma \backslash X)\to 0 \nonumber
    \end{align}
    where $\chi: \Z\to  \Z$ denotes the multiplication by $\chi$ map. The case of $|\chi| = 1$ is then straightforward. Assuming $|\chi| \neq 1$ and $\chi \neq 0$. From equation \eqref{eq-h-n-and-h-n-1} we get two short exact sequences

    $$0 \to H^{n-1}(\Gamma \backslash X) \to H^{n-1}(\Gamma \backslash G / M) \to \Z \to 0$$
    and
    $$0 \to \Z_{|\chi|} \to H^{n}(\Gamma \backslash G / M) \to H^1(\Gamma \backslash X) \to 0$$
    where $\Z_{|\chi|}$ is the cyclic group of order $|\chi|$. The sequences split since $\Z$ and $H^1(\Gamma \backslash X)$ are  free $\Z$-modules \footnote{We do not need any assumptions here, since if $Y$ is a topological space, then by the universal coefficient theorem 
$$H^1(Y, \Z) = Hom(H_1(Y, \Z), \Z)$$
which shows $H^1(Y, \Z)$ is always torsion free.}, hence

    $$H^n(\Gamma \backslash G / M) = H^1(\Gamma \backslash X) \oplus \Z_{|\chi|}$$
    $$H^{n-1}(\Gamma \backslash G / M)= H^{n-1}(\Gamma \backslash X) \oplus \Z.$$
    The case for $\chi = 0$ follows similarly.
\end{proof}

In case there is no torsion in $H^*(\Gamma \backslash G /M)$ and $\Gamma \backslash X$ is $Spin^c$, the K-theory of $C(G/P_0)\rtimes_r\Gamma$ could be calculated in the rank 1 case by employing Lemma \ref{lem-cohomology-unit-tangent-bundle} and Corollary \ref{cor-spinc-kk-equiv} together with the fact that the Chern map is injective when there is no torsion in cohomology, and thus implies an integral isomorphism. If there is torsion however, we will need to recall some properties of torsion in K-theory and how it relates to torsion in cohomology of the underlying space. 

\begin{lemma}
    \label{lem-torsion-in-k-theory}
    Let $Y$ be any compact CW-complex and let $H_t^{even}(Y)$, $H_t^{odd}(Y)$ denote the torsion part of the even and odd degree cohomology group respectively. Similarly we denote by $K^i_t(Y)$ the torsion part of the $i$'th K-group of $Y$. Then 
    
        $$|K^0_t(Y)| \leq |H_t^{even}(Y)| \quad \text{and} \quad  |K^1_t(Y)| \leq  |H_t^{odd}(Y)|.$$ 
    Furthermore the number of generators of $K_t^0(Y)$ and $K_t^1(Y)$ are less than or equal to the number of generators of $H_t^{even}(Y)$ and $H_t^{odd}(Y)$ respectively.
\end{lemma}
\begin{proof}
    This is a simple consequence of the Atiyah-Hirzebruch spectral sequence. We only prove the first part of the lemma for $K^0(Y)$. The other claims follow by a similar argument.
    
    Recall that the $E_2$-page of the Atiyah-Hirzebruch spectral sequence consists of elements of the form

        $$E_2^{p,q}= H^p(Y, K^q(\star))$$
    with the image of the differentials being torsion subgroups \cite[2.6]{dold_68}. It follows that the size of the torsion subgroup can only be reduced as we increase the page number. Thus $E_\infty^{p,q}$ has a torsion subgroup of order less than or equal to the torsion subgroup of $E_2^{p,q}$ and so
    
        $$E_\infty^{even}: = \bigoplus_{p\in \N} E_\infty^{p, p}$$
    has a smaller torsion subgroup than
        $$\bigoplus_{p\in \N}E_2^{p,p} = \bigoplus_{p \in \N}H^{p}(Y, K^p(\star)) =  H^{even}(Y).$$
    Now let $$0 = F^{n}K^0(Y) \subset F^{n-1}K^0(Y) \subset \dots \subset F^{0}K^0(Y)= K^0(Y)$$ 
    be the filtration of $K^0(Y)$ corresponding to the spectral sequence. The group $K^0(Y)$ is then calculated by successively solving the short exact sequences

    $$0\to F^{s+1}K^0(Y) \to  F^{s}K^0(Y) \to E_\infty^{s, s} \to 0.$$
    It is easy to check that 
    $$|F^{s}K^0_t(Y)| \leq  |F^{s+1}K^0_t(Y)| + |E_{\infty, t}^{s, s}|$$ where $E_{\infty, t}^{s,s}$ denotes the torsion subgroup of $E_{\infty}^{s,s}$. It follows by induction that 

    $$|K_t^0(Y)| \leq \left| \bigoplus_p E_{\infty, t}^{p,p} \right| \leq |H_t^{even}(Y)|.$$
\end{proof}

We can now calculate the K-theory of several rank 1 boundary C*-algebras explicitly in terms of the cohomology of the associated  locally symmetric spaces:

\begin{proposition}
    \label{prop-k-theory-boundary-algebras}
    Assume $X$ has rank 1 and assume $H^*(\Gamma \backslash X)$ is torsion free and $\Gamma \backslash X$ is $Spin^c$. Let $n = \mathrm{dim}(X)$ and denote by  $\chi= \chi(\Gamma \backslash X)$ the Euler characteristic. Then
    \begin{equation}
        \label{eq-boundary-cohomology-iso-0}
        K^0(C(\partial X)\rtimes_r \Gamma) = \bigoplus_{p\in \Z} H^{2p + n}(\Gamma \backslash G/M, \Z)
    \end{equation}
    
    \begin{equation}
        \label{eq-boundary-cohomology-iso-1}
        K^1(C(\partial X)\rtimes_r \Gamma) = \bigoplus_{p\in \Z} H^{2p+n +1}(\Gamma \backslash G/M, \Z)
    \end{equation}
    which can both be determined from $H^*(\Gamma \backslash X)$ by Lemma \ref{lem-cohomology-unit-tangent-bundle}. The class of the unit $[1_{C(\partial X) \rtimes_r\Gamma}] \in K^0(C(\partial X)\rtimes_r \Gamma)$ is a generator. Furthermore the unit is a torsion element (of order $|\chi|$) if and only if  $\chi \neq 0$.
\end{proposition}

\begin{proof}
    Assume $\chi = 0$. This is precisely the case where $H^*(\Gamma \backslash G/M)$ is torsion free (by Lemma \ref{lem-cohomology-unit-tangent-bundle})\footnote{note that $|\chi| \neq 1$. See the discussion after the proof of Proposition \ref{prop-k-theory-boundary-algebras}}. By Lemma \ref{lem-torsion-in-k-theory} and the fact that the Chern map is a rational isomorphism, we get\footnote{The isomorphism $K_0(C_0(\Gamma \backslash G/M))\otimes \Q = H^{even}(\Gamma \backslash G /M) \otimes \Q$ implies $K_0(C_0(\Gamma \backslash G/M))$ and $H^{even}(\Gamma \backslash G /M)$ have the same (finite) cardinality of torsion free generators . Assuming $H^{even}(\Gamma \backslash G /M)$ is torsion free, then  $K_0(C_0(\Gamma \backslash G/M))$ must be torsion free as well (by Lemma \ref{lem-torsion-in-k-theory}). They are thus two finitely generated free $\Z$-modules and isomorphic (as abstract modules) since the cardinality of their generators agree.}

    $$K_0(C_0(\Gamma \backslash G/M)) = H^{even}(\Gamma \backslash G /M)$$
    $$K_1(C_0(\Gamma \backslash G/M)) = H^{odd}(\Gamma \backslash G /M).$$
    The expression for the K-groups of $C(\partial X)\rtimes_r\Gamma$ then follow from Corollary \ref{cor-spinc-kk-equiv}.    

    Assume now  $\chi \neq 0,1$. This only happens when $n$ is even as $\chi= 0$ for all odd $n$ by Poincar\'{e} duality. By Corollary \ref{cor-spinc-kk-equiv}
    $$K_*(C(\partial X) \rtimes_r\Gamma)= K^*(\Gamma \backslash G /M).$$
    Now by Lemma \ref{lem-torsion-in-k-theory}, since the order of the torsion subgroup of $K^0(\Gamma \backslash G /M)$ and the number of its generators is less than or equal to those of $H^{even}(\Gamma \backslash G/M)$ we see that 
    $K^0(\Gamma \backslash G /M) = \Z^r \oplus\Z_{s}$
    where $s\leq |\chi|$ and $r = \mathrm{rank}(H^{even}(\Gamma \backslash G/M))$. Since $\Gamma$ is a lattice in a rank 1 Lie group it satisfies the Baum-Connes conjecture with coefficients \cite[Theorem 6.15]{aparicio_baum-connes_2019} so by  \cite[Corollary 2]{emerson_euler_2006}, we know that the class of the unit $[1_{C(\partial X)\rtimes_r\Gamma}] \in K^0(C(\partial X)\rtimes_r \Gamma)$ has order $|\chi|$. This forces us to set $s = |\chi|$, and thus $[1_{C(\partial X)\rtimes_r\Gamma}]$ generates the torsion part of $K^0(C(\partial X)\rtimes_r \Gamma)$.

    To see that $[1_{C(\partial X)\rtimes_r\Gamma}]$ is a generator also when $\chi = 0$, we employ \cite[Corollary 32]{emerson_euler_2006}, where the authors produces an exact sequence
    \begin{equation}
        \label{eq-meyer-emerson}
        0\to K_0(C^*_r(\Gamma)) \xrightarrow[]{u_*}K_0(C(\partial X)\rtimes_r\Gamma) \to K_1(C_0(\Gamma \backslash X)) \to 0.
    \end{equation}
    Our assumptions imply that all the groups in this extension are torsion free, thus $u_*$ maps generators to generators. It follows that the unit in $K_0(C(\partial X)\rtimes_r\Gamma)$ is a generator if and only if $[1_{C^*_r(\Gamma)}] \in K_0(C^*_r(\Gamma))$ is a generator. But this is always true for any discrete group since the canonical positive faithful trace 
    $$\tau: C_r^*(\Gamma) \to \C$$
    determines a left inverse of the unital inclusion $\C\to C_r^*(\Gamma)$.
\end{proof}

It is known (though we have not found a reference) that $|\chi(\Gamma \backslash X)| \neq 1$ for oriented closed locally symmetric spaces $\Gamma \backslash X$. If such a space existed we would get by \cite[Corollary 32]{emerson_euler_2006} a short exact sequence

    $$ 0\to \langle [1_{C^*_r(\Gamma)}] \rangle \to K_0(C^*_r(\Gamma)) \xrightarrow[]{u_*}K_0(C(\partial X)\rtimes_r\Gamma) \to K_1(C_0(\Gamma \backslash X)) \to 0$$
where the inclusion on the left is the inclusion of the subgroup generated by $ [1_{C^*_r(\Gamma)}]$ in $K_0(C^*_r(\Gamma))$ and 

    $$u: C_r^*(\Gamma) \to C(\partial X) \rtimes_r\Gamma$$
is the (unital) inclusion map. Thus $[1_{C(\partial X)\rtimes_r\Gamma}] = u_*([1_{C^*_r(\Gamma)}])$, but this is zero by exactness of the sequence. Since our boundary C*-algebra is KK-equivalent to a (non-KK-contractible) commutative C*-algebra we would get a contradiction.

\begin{proposition}
    \label{prop-unit-and-volume}
    Let $\Gamma, \Gamma' \subset X$ be two cocompact lattices of a rank 1 symmetric space whose associated locally symmetric space is $Spin^c$. Assume $\chi^{(')} :=  \chi(X/\Gamma^{(')})  \neq 1$. Then
    $$C(\partial X)\rtimes_r\Gamma = C(\partial X)\rtimes_r\Gamma'$$
    implies 
        $$\chi = \chi' \qquad \text{and } \qquad Vol(X/\Gamma) = Vol(X/\Gamma').$$
\end{proposition}

\begin{proof}
    Since $|\chi|$ is the order of the unit in $K(C(\partial X)\rtimes_r \Gamma)$, the fact that $\chi = \chi'$ follows from the classification theorem of Kirchberg algebras.

    The final claim follows from the work of Eberlein and Chen in \cite[Cor.1 (4)]{chen_isometry_1982}.
\end{proof}

Note that if $\Gamma$ is as in Proposition \ref{prop-unit-and-volume} and $G$ has no factors isomorphic to $PSL_2(\R)$ or $PSL_2(\C)$, then by a result of Wang \cite{wang} there are only finitely many lattices $\Gamma'$ up to conjugacy for which 
$$Vol(\Gamma \backslash X) = Vol(\Gamma'\backslash X).$$
Thus, there can only be finitely many isomorphic maximal boundary C*-algebras arising in this way.

We end with a family of examples of isomorphic boundary C*-algebras of distinct locally symmetric spaces. 

\begin{example}[homology spheres]
    An integral homology $n$-sphere is a closed smooth manifold $M$ which has the integral (co)homology of and $n$-sphere. Thus 
    $$\chi(M)= \begin{cases} 2 & \text{ if n is even} \\ 0 & \text{if n is odd.} \end{cases}$$
    There are infinitely many hyperbolic homology  $3$-spheres \cite{hom_note_2018}, given by quotients of torsion free lattices in $\H_\R^3$. Since $\chi = 0$ for all homology $3$-spheres, by Proposition \ref{prop-k-theory-boundary-algebras} these boundary C*-algebras are all isomorphic. Similar examples can be constructed using homology $2$-spheres and homology $4$-spheres (both of which are $Spin^c$ since the third cohomology group vanishes) showing that the dimension of $X$ is not an invariant of the maximal boundary C*-algebras.
\end{example}

\bibliographystyle{alpha.bst}

\end{document}